\documentclass[11pt]{amsart}
\usepackage{amsmath,amssymb,mathrsfs,color}

\def\R{{\mathbb R}}

\newtheorem{thm}{Theorem}[section]

\newtheorem{lem}[thm]{Lemma}
\newtheorem{prop}[thm]{Proposition}

\theoremstyle{definition}
\newtheorem{de}[thm]{Definition}
\theoremstyle{remark}
\newtheorem{rem}[thm]{Remark}
\newtheorem{exam}[thm]{Example}
\numberwithin{equation}{section}

\allowdisplaybreaks

\begin{document}

\title[Existence, uniqueness and ergodicity for MVSDEs]
{Existence, uniqueness and ergodicity for McKean-Vlasov SDEs under distribution-dependent Lyapunov conditions}

\author{Zhenxin Liu}
\address{Z. Liu: School of Mathematical Sciences,
Dalian University of Technology, Dalian 116024, P. R. China}
\email{zxliu@dlut.edu.cn}

\author{Jun Ma}
\address{J. Ma (Corresponding author): School of Mathematics and Statistics,
Northeast Normal University, Changchun, 130024, P. R. China}
\email{mathmajun@163.com}


\date{}

\subjclass[2010]{60H10, 37H30, 60J60}

\keywords{McKean-Vlasov SDEs, distribution-dependent Lyapunov conditions,
ergodicity, Lions derivation}

\begin{abstract}
In this paper, we prove the existence and uniqueness of solutions as well as ergodicity for McKean-Vlasov SDEs under Lyapunov conditions,
in which the Lyapunov functions are defined on $\mathbb R^d\times \mathcal P_2(\mathbb R^d)$,
i.e. the Lyapunov functions depend not only on space variable but also on distribution variable.
It is reasonable and natural to consider distribution-dependent Lyapunov functions
since the coefficients depends on distribution variable.
We apply the martingale representation theorem and a modified Yamada-Watanabe theorem to obtain the existence and uniqueness of solutions.
Furthermore, the Krylov-Bogolioubov theorem is used to get ergodicity
since it is valid by linearity of the corresponding Fokker-Planck equations on $\mathbb R^d\times \mathcal P_2(\mathbb R^d)$.
In particular, if the Lyapunov function depends only on space variable,
we obtain exponential ergodicity for semigroup $P_t^*$ under Wasserstein quasi-distance.
Finally, we give some examples to illustrate our theoretical results.
\end{abstract}

\maketitle

\section{Introduction}
McKean-Vlasov stochastic differential equations (MVSDEs) are a special class of SDEs,
whose coefficients depend not only on the microcosmic site,
but also on the macrocosmic distribution of particles
\begin{equation}\label{main}
dX_t= b(t,X_t,\mathcal{L}_{X_t})dt+ \sigma(t,X_t,\mathcal{L}_{X_t})dW_t,
\end{equation}
where $\mathcal{L}_{X_t}$ denotes the law of $X_t$.
With the increasing demands on practical financial markets and social systems,
MVSDEs have drawn much attention.
One notable example is the impact of macrocosmic distribution on the rate of price changes in a financial market.

MVSDEs are also known as mean-field SDEs or distribution-dependent SDEs,
which are used to study the interacting particle systems and mean-field games.
Such SDEs were first studied by McKean \cite{Mckean} which showed the propagation of chaos in
physical systems of N-interacting particles related to the Boltzmann equations
and was inspired by Kac's work on the Vlasov kinetic equations \cite{Kac}.
Sznitman in \cite {Sznitman1,Sznitman2} proved the propagation of chaos
and the limit equation under the globally Lipschitz condition.
The limit equation can be described as an evolution equation
called the aforementioned MVSDE.
Owing to its importance and reality, MVSDEs are studied extensively.
Larsy and Lions \cite{Larsy-Lions1,Larsy-Lions2,Larsy-Lions3} introduced
mean-field games in order to study large population deterministic and stochastic differential
games which were independent of the work of Huang, Malham\'e and Caines \cite{HMC1,HMC2}.
The study of these two groups has significantly attracted the interest of more researchers to MVSDEs.
Wang \cite{Wang_18} obtained the well-posedness and exponential ergodicity under monotone conditions.
In \cite{RTW}, Ren et al. proved the
existence and uniqueness of solutions in infinite dimension under a Lyapunov condition.
In \cite{DQ}, Ding and Qiao established weak existence, pathwise uniqueness and optimal
strong convergence rate by the Euler-Maruyama approximation. In \cite{LM}, Liu and Ma showed the existence and uniqueness
of solutions and invariant measures under Lyapunov conditions which are different from that of \cite{RTW}.
Hong at al. \cite{HHL} proved the strong and weak well-posedness
for a class of McKean-Vlasov stochastic (partial) differential equations under locally monotone conditions.
Buckdahn et al. \cite{BLPR} investigated the relationship between the functional of the form $Ef(t, \bar X_t, \mathcal L_{X_t})$ and the associated second-order PDE, involving derivative with respect to the law.

This paper is dedicated to studying the existence and uniqueness of solutions as well as ergodicity
under Lyapunov conditions
in which the Lyapunov functions are defined on $\mathbb R^d\times \mathcal P_2(\mathbb R^d)$,
i.e. the Lyapunov functions depend not only on space variable but also on distribution variable.
For the definition of $\mathcal P_2(\mathbb R^d)$, see Section 2 for details.
Considering distribution-dependent Lyapunov functions is reasonable and natural
since the coefficients of MVSDEs depend on distribution variable.
However, the existing study under Lyapunov conditions for MVSDEs focuses on the space-dependent Lyapunov functions,
which are independent of distribution variable, such as \cite{BRS,RTW,LM}.
In \cite{BRS}, Bogachev at al. obtained the convergence in variation of probability solutions of
nonlinear Fokker-Planck equations to stationary solutions under space-dependent Lyapunov conditions.
In \cite{RTW}, Ren at al. showed the existence and uniqueness of solutions
for McKean-Vlasov stochastic partial differential equations under space-dependent Lyapunov conditions.
In \cite{LM}, Liu and Ma get the existence and uniqueness as well as exponential ergodicity for MVSDEs under space-dependent Lyapunov conditions.
In this paper, we adopt Khasminskii's approach \cite{Khasminskii}
to establish the existence and uniqueness of solutions, and employ the Krylov-Bogolioubov theorem to get ergodicity under distribution-dependent Lyapunov conditions,
in which the Lyapunov functions depend not only on space variable but also on distribution variable.

The first main result of this paper is  the existence and uniqueness of solutions for MVSDEs under distribution-dependent Lyapunov conditions.
Since the coefficients depend on the distribution which is a global property,
the classical truncation in the space variable does not work in this situation.
Following Ren et al. \cite{RTW}, we need to truncate the equation in both space and distribution variables simultaneously
to overcome this difficulty. However, this is not enough to obtain the existence of solutions here,
which is different from that of \cite{LM,RTW}. Therefore, the condition (H4) in Section 3 is needed
in order to guarantee that there exists a tight sequence on $C[0,T]$.
Thus, the existence of weak solutions follows by Skorokhod's representation theorem and the martingale representation theorem.
And applying a modified Yamada-Watanabe theorem (see \cite[Lemma 2.1]{HW}), we obtain the existence and uniqueness of solutions.
Moreover, the condition (H4) ensures that the solutions have finite second-order moments.
Consequently, It\^o's formula with Lions derivation holds since the Lions derivation is defined on $\mathcal P_2(\mathbb R^d)$.

The second main result of this paper is ergodicity for MVSDEs under distribution-dependent Lyapunov conditions in L\'evy-Prohorov distance.
We apply the Krylov-Bogolioubov theorem to obtain the existence of invariant measures,
which is valid since the corresponding Fokker-Planck equations on $\mathbb R^d\times \mathcal P_2(\mathbb R^d)$ are linear.
In order to illustrate linearity of the Fokker-Planck equations,
we introduce the operators $P_t$ acting on $C_b(\mathbb R^d\times \mathcal P_2(\mathbb R^d))$ and $P_t^*$ on $\mathcal P(\mathbb R^d\times \mathcal P_2(\mathbb R^d))$
which are semigroups and conjugate to each other;
see Section 4 for details.
Particularly, if the Lyapunov function
depends only on space variable, we get exponential ergodicity for semigroup $P_t^*$
under Wasserstein quasi-distance induced by the Lyapunov function.
Especially, this result reduces to Theorem 6.1 \cite{RRW} when the Lyapunov function is equal to the square of norm.

Considering ergodicity on the product space $\mathbb R^d\times \mathcal P_2(\mathbb R^d)$
leads to the linearity of the corresponding Fokker-Planck equations,
and the existence of semigroups of operators $P_t,P_t^*$ which are conjugate to each other.
However, if we restrict our analysis to $\mathbb R^d$ alone, these results do not hold.
Indeed, the Fokker-Planck equations on $\mathbb R^d$ are nonlinear and the operator $P_t$ is not a semigroup
due to their coefficients depending on the distribution; see \cite{Wang_18} for details. Consequently, the Krylov-Bogolioubov theorem is invalid in this situation. As a result, the classical method fails to establish the existence of invariant measures.
Therefore, Wang \cite{Wang_18} applied monotone conditions to obtain exponential ergodicity, which avoids this problem.
Similarly, \cite{BRS,LM} employed Lyapunov conditions to address this issue.
In this paper, we consider the corresponding Fokker-Planck equations on $\mathbb R^d\times \mathcal P_2(\mathbb R^d)$,
which are linear equations, to overcome this problem.
Moreover, to define semigroups $P_t$ and $P_t^*$, we consider the coupled SDEs,
which is similar to \cite{BLPR,RRW}.
The reason for considering these coupled SDEs is that
for given $f\in C_b(\mathbb R^d\times \mathcal P_2(\mathbb R^d))$,
$(x,\mu)\in \mathbb R^d\times \mathcal P_2(\mathbb R^d)$ in $P_tf(x,\mu)$ is not constrained to be a `pair',
meaning that $\mu$ represents an arbitrary probability measure and may not necessarily equal $\delta_x$.
By these coupled SDEs, we define the semigroups $P_t,P_t^*$, and obtain linearity of the Fokker-Planck equation.
Meanwhile, we get exponential ergodicity for $P_t^*$ under the Lyapunov condition,
in Wasserstein quasi-distance induced by the Lyapunov function.

The rest of this paper is arranged as follows.
In Section 2, we collect a number of preliminary results
concerning Lions derivation, chain rule, It\^o's formula and optimal transportation cost, etc.
Section 3 presents the existence and uniqueness of solutions under distribution-dependent Lyapunov conditions.
Section 4 establishes ergodicity
under distribution-dependent Lyapunov conditions in L\'evy-Prohorov distance,
and exponential ergodicity for semigroup $P_t^*$
under Lyapunov conditions in Wasserstein quasi-distance.
In Section 5, we provide some examples to illustrate our theoretical results.

\section{Preliminaries}

Throughout the paper, let $(\Omega, \mathcal F, \{\mathcal F_t\}_{t\ge 0},P)$ be a filtered complete probability space,
in which the filtration $\{\mathcal F_t\}$ is assumed to satisfy the usual condition, i.e. it is right continuous and $\mathcal F_0$ contains all $P$-null sets. Let $W$ be an $n$-dimensional Brownian motion defined on $(\Omega, \mathcal F, \{\mathcal F_t\}_{t\ge 0}, P)$.  We denote by $B^\top$ the transpose of matrix $B\in \mathbb R^{n_1\times n_2}$ with $n_1,n_2\geq1$, $tr(B)$ the trace of $B$ and $|B|:=\sqrt {tr(B^\top B)}$ the norm of $B$.
$\mathcal P(\mathbb R^d)$ denotes the space of probability measures on $\mathbb R^d$ and $\mathcal P_{p}(\mathbb R^d):= \{\mu\in \mathcal P(\mathbb R^d): \int_{\mathbb R^d}|x|^p\mu(dx)<\infty\}$ with $p\ge 1$. Let us introduce some basic notations as follows.

The MVSDE can be described as follows:
\begin{align}
\left\{
\begin{aligned}\label{MVSDE}
&dX_t= b(t,X_t,\mathcal{L}_{X_t})dt+ \sigma(t,X_t,\mathcal{L}_{X_t})dW_t\\
&X_0= \xi,
\end{aligned}
\right.
\end{align}
where $b:[0,\infty) \times \mathbb R^d \times \mathcal P (\mathbb R ^d)\to \mathbb R^d$, $\sigma:[0,\infty) \times \mathbb R^d \times \mathcal P (\mathbb R ^d)\to \mathbb R^{d\times n}$, and
$X_0$ is $\mathcal F_0$-measurable and satisfies some integrable condition to be specified below.
We next introduce the infinitesimal generator $\mathcal L$ of the above MVSDE, which involves Lions derivation, differentiable space and It\^o's formula.
The definition of Lions derivation is brought by Lions \cite{Lions}, which is revised by Cardaliaguet \cite{Cardaliaguet}.
And the differentiable space $C^{(1,1)}(\mathcal P_2(\mathbb R^d))$ is called partially $C^2$ in Chassagneux at al. \cite{CCD}
\begin{de}[Lions derivation]
A function $f:\mathcal P_2(\mathbb R^d)\to\mathbb R$ is called differentiable at $\mu\in\mathcal P_2(\mathbb R^d)$, if there exists a random variable $X\in L^2(\Omega)$ with $\mu=\mathcal L_X$ such that $\tilde f(X):=f(\mathcal L_X)$ and $\tilde f$ is Fr\'echet differentiable at $X$.
$f$ is called differentiable on $\mathcal P_2(\mathbb R^d)$ if $f$ is differentiable at any $\mu\in\mathcal P_2(\mathbb R^d)$.
\end{de}
\begin{de}
The space $C^{(1,1)}(\mathcal P_2(\mathbb R^d))$ contains the functions $f:\mathcal P_2(\mathbb R^d)\to\mathbb R$
satisfying the following conditions:
(i) $f$ is differentiable and its deviation $\partial_{\mu} f(\mu)(y)$ has a continuous version in $(\mu,y)$, still denoted $\partial_{\mu} f(\mu)(y)$;
(ii) the map $(\mu,y)\mapsto \partial_{\mu} f(\mu)(y)$ is locally bounded;
(iii) $(\mu,y)\mapsto \partial_{\mu} f(\mu)(y)$ is jointly continuous at any $(\mu,y)$ with $y\in supp(\mu)$;
(iv) $\partial_{\mu} f(\mu)(\cdot)$ is continuously differentiable for any $\mu$,
its derivation $\partial_y \partial_{\mu} f(\mu)(y)$ is locally bounded in $(\mu,y)$
and is jointly continuous at any $(\mu,y)$ with $y\in supp(\mu)$.
\end{de}
We then introduce the chain rule which is established based on Chassagneux at al. \cite[Theorem 16]{CCD},
in conjunction with the argument from \cite{BLPR}.
And see \cite{BLPR} for details involving copy in the following lemma.
\begin{lem}[Chain rule]
Let $b,\sigma$ be progressively measurable with
\begin{equation}\label{ConditionIto}
E\int_0^T|b(s)|^2+|\sigma(s)|^4dt<\infty, \quad T\in[0,\infty].
\end{equation}
Then for any $f\in C^{(1,1)}(\mathcal P_2(\mathbb R^d))$ with
$$
\sup_{\mu\in K}\int|\partial_{\mu} f(\mu)(y)|^2+ |\partial_y \partial_{\mu} f(\mu)(y)|^2\mu(dy)<\infty
$$
where $K\subset \mathcal P_2(\mathbb R^d)$ is a compact set,
we have
\begin{align*}
f(\mathcal L_{X_t})=&f(\mathcal L_{X_0})+\tilde E\int_0^t b(s,\tilde X_s,\mathcal L_{X_s})\cdot\partial_{\mu} f(\mathcal L_{X_s})(\tilde X_s)\\
&+
\frac 12tr\big((\sigma\sigma^{\top})(s,\tilde X_s,\mathcal L_{X_s})\cdot\partial_y \partial_{\mu} f(\mathcal L_{X_s})(\tilde X_s)\big)ds.
\end{align*}
Here $\tilde X_s$ denotes the solution defined on another probability space $(\tilde\Omega, \tilde {\mathcal F},\tilde P)$
which is a copy of $(\Omega, \mathcal F, P)$,
$\tilde E$ denotes the expectation on $(\tilde\Omega, \tilde {\mathcal F},\tilde P)$.
\end{lem}
We next give the definition of differentiable space $C^{2,(1,1)}(\mathbb R^d\times \mathcal P_2(\mathbb R^d))$
and introduce It\^o's formula on $\mathbb R^d\times \mathcal P_2(\mathbb R^d)$,
which is obtained by the approach \cite{BLPR} explaining the transition from It\^o's formula for functions of measures only, i.e. chain rule, to the general case.
\begin{de}
The space $C^{2,(1,1)}(\mathbb R^d\times \mathcal P_2(\mathbb R^d))$ contains the function $f:\mathbb R^d\times\mathcal P_2(\mathbb R^d)\to\mathbb R$
satisfying the following conditions:
$f(\cdot, \mu)\in C^2(\mathbb R^d)$ for any $\mu\in\mathcal P_2(\mathbb R^d)$,
and $f(x,\cdot)\in C^{(1,1)}(\mathcal P_2(\mathbb R^d))$ for any $x\in \mathbb R^d$.
\end{de}
\begin{prop}[It\^o's formula]
Assume
\begin{equation*}
E\int_0^T|b(s)|^2+|\sigma(s)|^4dt<\infty, \quad T\in[0,\infty),
\end{equation*}
and $V\in C^{2,(1,1)}(\mathbb R^d\times \mathcal P_2(\mathbb R^d))$ with
$$
\sup_{\mu\in K}\int|\partial_{\mu} V(x,\mu)(y)|^2+ |\partial_y \partial_{\mu}V(x,\mu)(y)|^2\mu(dy)<\infty,
$$
where $K\subset \mathcal P_2(\mathbb R^d)$ is a compact set.
Then we have
\begin{align*}
V(X_t,\mathcal L_{X_t})
=&{}
V(X_0,\mathcal L_{X_0})+ \int_0^t\bigg(b(s,X_s,\mathcal L_{X_s})\cdot \partial_x V(X_s,\mathcal L_{X_s})\\
&\quad + \frac 12tr((\sigma\sigma^{\top})(s,X_s,\mathcal L_{X_s})\cdot\partial_x^2 V(X_s,\mathcal L_{X_s}))\bigg) ds\\
&+{}
\tilde E\int_0^t \bigg(b(s,\tilde X_s,\mathcal L_{X_s})\cdot\partial_{\mu} V(X_s,\mathcal L_{X_s})(\tilde X_s)\\
&\quad+{}
\frac 12tr((\sigma\sigma^{\top})(s,\tilde X_s,\mathcal L_{X_s})\cdot\partial_y \partial_{\mu} V(X_s,\mathcal L_{X_s})(\tilde X_s))\bigg)ds\\
&+{}
\int_0^t\sigma(s,X_s,\mathcal L_{X_s})\cdot \partial_x V(X_s,\mathcal L_{X_s})dW_s\\
=&:{}V(X_0,\mathcal L_{X_0})+ \int_0^t \mathcal LV(X_s,\mathcal L_{X_s})ds+\int_0^t\sigma(s,X_s,\mathcal L_{X_s})\cdot \partial_x V(X_s,\mathcal L_{X_s})dW_s,
\end{align*}
where
\begin{align}\label{Itoformula}
\mathcal LV(X_s,\mathcal L_{X_s}):={}&b(s,X_s,\mathcal L_{X_s})\cdot \partial_x V(X_s,\mathcal L_{X_s})+ \frac 12tr((\sigma\sigma^{\top})(s, X_s,\mathcal L_{X_s})\cdot\partial_x^2 V(X_s,\mathcal L_{X_s}))\\\nonumber
&+{} \tilde E \bigg[b(s,\tilde X_s,\mathcal L_{X_s})\cdot\partial_{\mu} V(X_s,\mathcal L_{X_s})(\tilde X_s)\\\nonumber
&\quad{} + \frac 12tr((\sigma\sigma^{\top})(s,\tilde X_s,\mathcal L_{X_s})\cdot\partial_y \partial_{\mu} V(X_s,\mathcal L_{X_s})(\tilde X_s)\bigg].
\end{align}
\end{prop}

In the time homogeneous case, i.e. $b$ and $\sigma$ do not depend on $t$, a measure $\mu \in \mathcal P(\mathbb R^d)$ is said to be an invariant measure for MVSDE \eqref{MVSDE} if $\mathcal L_{X_t}= \mu$ for all $t\geq 0$, and the equation is said to be ergodic if there exists $\mu_{\mathcal I} \in \mathcal P(\mathbb R^d)$ such that $\mathcal L_{X_t}\to \mu_{\mathcal I}$ weakly as $t\to\infty$.
It is obvious that ergodicity implies the uniqueness of invariant measure.

Define Wasserstein distance $W_p$ on $\mathcal P_p(\R^d)$ with $p\ge 1$ as follows:
$$W_p(\mu, \nu):=\inf_{\pi \in \mathcal C(\mu, \nu)} \bigg(\int_{\mathbb R^d\times\R^d} |x-y|^p \pi(dx, dy)\bigg)^{1/p}$$
for $\mu, \nu\in \mathcal P_p(\R^d)$,
where $\mathcal C(\mu, \nu)$ is the set of couplings between $\mu$ and $\nu$. As usual, we also denote $\mu(f):=\int_{\R^d} f(x)\mu(dx)$ in what follows for any function $f$ defined on $\R^d$ and $\mu\in\mathcal P(\R^d)$.

We also define an another distance $\omega:\mathcal P(\mathbb R^d)\times\mathcal P(\mathbb R^d)\to\mathbb R^+$,
which is called L\'evy-Prohorov distance, by
$$
\omega(\mu,\nu)=\inf\{\delta:\mu(A)<\nu(A^{\delta})+\delta,\nu(A)<\mu(A^{\delta})+\delta \quad \hbox{\text{for} }\hbox{\text{all} } \hbox{\text{closed} } \hbox{\text{set} } A\in\mathcal F\},
$$
where $A^{\delta}=\{x: \exists y\in A ~\hbox{\text{such} } \hbox{\text{that} }~|x-y|<\delta\}$.
Convergence under the L\'evy-Prohorov distance is equivalent to the weak convergence of measures, see for instance \cite[Theorem 1.11]{Prohorov}.
And we present a proposition that characterizes the relationship between L\'evy-Prohorov and Wasserstein distance
on $\mathcal P(\mathbb R^d)$ as follows; see \cite[Lemma 5.3, Theorem 5.5]{Chen} or \cite[Theorem 6.9]{Villani}.
\begin{prop}\label{Ww}
For any $\mu,\nu\in\mathcal P(\mathbb R^d)$, we have
\begin{enumerate}
\item $W_p(\mu,\nu)\ge \omega(\mu,\nu)^{1+\frac 1p}.$
\item A subset $\mathcal M\subset\mathcal P(\mathbb R^d)$ is compact in $W_p$ if and only if
$\mathcal M$ is weakly compact, and
\begin{align*}
\lim_{N\to\infty}\sup_{\mu\in\mathcal M}\int_{\{x:|x-x_0|>N\}}|x-x_0|^p\mu(dx)=0,
\end{align*}
for some $x_0\in\mathbb R^d$.
\end{enumerate}
\end{prop}

Let $\rho:(\mathbb R^d\times\mathcal P_2(\mathbb R^d))\times(\mathbb R^d\times\mathcal P_2(\mathbb R^d)) \to \mathbb R^+$ be a distance-like function satisfying $x=y,\mu=\nu$ if and only if $\rho((x,\mu),(y,\nu))=0$.
Let $\mathcal P(\mathbb R^d\times\mathcal P_2(\mathbb R^d))$ denote the space of all probability measures on $\mathbb R^d\times\mathcal P_2(\mathbb R^d)$.
We now introduce the Wasserstein quasi-distance based on $\rho$. For any $\mu, \nu \in \mathcal P_{\rho}(\mathbb R^d\times\mathcal P_2(\mathbb R^d)):=\{\Lambda\in \mathcal P(\mathbb R^d\times\mathcal P_2(\mathbb R^d)): \int\rho((x,\mu),(0,\delta_0))\Lambda(dx,d\mu)<\infty\}$, let
\begin{align*}
W_{\rho}(\mu, \nu):=\inf_{\pi \in \mathcal C(\mu, \nu)} \int_{(\mathbb R^d\times\mathcal P_2(\mathbb R^d)) \times (\mathbb R^d\times\mathcal P_2(\mathbb R^d))}\rho(x,y) \pi(dx, dy)=\inf E \rho(X, Y),
\end{align*}
where $\mathcal C(\mu, \nu)$ is the set of couplings between $\mu$ and $\nu$, and the second infimum is taken over all random variables $X, Y$ on $\mathbb R^d\times \mathcal P_2(\mathbb R^d)$ whose laws are $\mu, \nu$ respectively.
In general, $W_{\rho}$ is not a distance since the triangle inequality may not hold.
But it is complete in the sense that any $W_{\rho}$--Cauchy sequence in $\mathcal P_{\rho}(\mathbb R^d\times\mathcal P_2(\mathbb R^d))$ is convergent,
i.e. for any Cauchy sequence $\{\mu_n\} \subset \mathcal P_{\rho}(\mathbb R^d\times\mathcal P_2(\mathbb R^d))$,
there exists a measure $\mu \in \mathcal P_{\rho}(\mathbb R^d\times\mathcal P_2(\mathbb R^d))$ such that $W_{\rho}(\mu_n, \mu) \to 0$ as $n\to \infty$.
When $\big((x,\mu),(y,\nu)\big)\mapsto\rho\big((x,\mu),(y,\nu)\big)$ is a distance,
$W_{\rho}$ satisfies the triangle inequality and is hence a distance on $\mathcal P_{\rho}$.
This should not cause confusion with $W_p$ and $\mathcal P_p(\mathbb R^d)$ introduced above.
\section{Existence and uniqueness of solutions}

In this section, we consider the existence and uniqueness of solutions for MVSDEs under distribution-dependent Lyapunov conditions.
In order to obtain the desired results, we get a tight sequence by the Lyapunov condition,
and then apply the martingale representation theorem as well as a modified Yamada-Watanabe theorem.
Now, we introduce some assumptions for equation \eqref{MVSDE}.
\begin{enumerate}

\item [(H1)]For any $N\geq 1$, there exists a constant $ C_N\geq 0 $ such that for any $ |x|, |y|\leq N$ and  ${\rm supp} \mu, {\rm supp} \nu \subset B(0,N)$ we have
\begin{align*}
&|b(t, x, \mu) |+ |\sigma(t, x, \mu)| \leq C_N,\\
&| b(t, x, \mu)- b(t, y, \nu) |+ |\sigma(t, x, \mu)- \sigma(t, y, \nu)| \leq  C_N\big(| x- y | + W_2(\mu,\nu)\big).
\end{align*}
Here $B(0,N)$ denotes the closed ball in $\R^d$ centered at the origin with radius $N$.

\item [(H2)](Lyapunov condition)
There exists a nonnegative function $ V\in C^{2,(1,1)}(\mathbb R^d \times \mathcal P_2(\mathbb R^d))$ such that there exist nonnegative constants $ \lambda,\tilde\lambda  \in \mathbb R $ satisfying for all $(x,\mu)\in \mathbb R^d \times \mathcal P_2(\R^d)$
\begin{align*}
\mathcal LV(x,\mu)  &\leq  \lambda V(x, \mu) + \tilde\lambda \int _{\mathbb R^d} V(x, \mu)\mu(dx), \\
V_R&:= \inf_{|x|\geq R}V(x, \mu)\to \infty ~as~ R\to \infty,
\end{align*}
where $\mathcal L$ denotes the generator.

\item [(H3)](Continuity) For any bounded sequences $ \{{x_n, \mu_n}\} \in \mathbb R^d \times \mathcal P_V (\mathbb R^d) $ with $ x_n \to x $ and $\mu_n \to \mu$ weakly in $\mathcal P (\mathbb R^d)$ as $ n\to \infty $, we have
\begin{equation}\nonumber
\lim_{n\to \infty} \sup_{t\in [0,T]}{ | b(t, x_n, \mu _n)- b(t, x, \mu) |+ |\sigma(t, x_n, \mu _n)- \sigma(t, x, \mu)|}= 0.
\end{equation}
where $ \mathcal P_V (\mathbb R^d):= \left\{ \mu \in \mathcal P(\mathbb R^d):\int_{\mathbb R^d} V(x,\mu)\mu(dx) <\infty\right\}$.

\item[(H4)] There exist constants $\ell>1,K>0$ such that
\begin{align*}
|b(t,x,\mu)|^{2\ell}+|\sigma(t,x,\mu)|^{2\ell} \le K(1+V(x,\mu)).
\end{align*}

\item [(H5)]There exist constants $M,\epsilon >0 $ and increasing unbounded function $L: \mathbb N \to (0,\infty)$ such that for any $N\geq 1,|x|\vee |y| \leq N $ and $  \mu,\nu \in \mathcal P (\mathbb R ^d ) $ satisfying
\begin{align*}
&| b(t, x, \mu)- b(t, y, \nu) |+ |\sigma(t, x, \mu)- \sigma(t, y, \nu)|\\
&\leq L_N(| x- y |+ W_{2,N}(\mu , \nu)+ Me^{-\epsilon L_N}(1\wedge W_2(\mu,\nu))),
\end{align*}
where
\begin{align*}
W_{2,N}^2(\mu,\nu):= \inf_{\pi \in \mathcal C(\mu,\nu)} \int_{\mathbb R^d \times \mathbb R^d}|\phi_N(x) -\phi_N(y)|^2 \pi(dx,dy) ,
\phi_N (x):= \frac{Nx}{N\vee |x|}.
\end{align*}
\end{enumerate}
\begin{rem}\rm
\begin{enumerate}
\item  The condition (H4) guarantees the existence of tight sequences as well as solutions with finite second-order moments.
\item  If the function $L:\mathbb N\to (0,\infty)$ in (H5) is bounded, i.e. $b$ and $\sigma$ are globally Lipschitz, $M$ should be $0$. This then reduces to the case of Lipschitz condition, so we assume that the function $L$ is unbounded in (H5).
\end{enumerate}
\end{rem}
\begin{thm}\label{Thm1}
Assume (H1)--(H5). Then for any $T>0, X_0\in L^2(\Omega, \mathcal F, P)$, equation \eqref{MVSDE} has a solution $X_.$
which satisfies
\begin{equation}\label{ES}
EV(X_t,\mathcal L_{X_t})\le e^{(\lambda+ \tilde\lambda)t}EV(X_0,\mathcal L_{X_0}),\quad\hbox{for } t\in [0,T].
\end{equation}
Moreover, if (H5) holds, the solution is unique.
\end{thm}

\begin{proof}
If the Lyapunov function $V$ is independent of $\mu$, this degenerates into the situation in \cite[Theorem 3.3]{LM}.
And the proof of uniqueness is omitted since it is the same as \cite[Proposition 2.1]{RTW} and \cite[Theorem 3.3]{LM}.
Therefore we only prove the existence of solution for equation \eqref{MVSDE} under the distribution-dependent Lyapunov condition.

We truncate the coefficients $b$ and $\sigma$ as follows: for any $n\geq 1, t\in [0,T], x\in \mathbb R^d$ and $\mu \in \mathcal P(\mathbb R^d)$, define
\begin{align*}
b^n(t, x, \mu):= b(t, \phi_n(x), \mu \circ \phi_n),\\
\sigma^n(t, x, \mu):= \sigma(t, \phi_n(x), \mu \circ \phi_n).
\end{align*}
Thus for each $n\geq 1$, $b^n, \sigma^n$ are Lipschitz and satisfy the linear growth condition.
Therefore, by \cite[Theorem 2.1]{Wang_18} or \cite[Theorem 3.1]{LM}, the equation
\begin{equation}\label{TMVSDE}
\left\{
\begin{aligned}
dX^n_t&= b^n(t,X^n_t,\mathcal{L}_{X^n_t})dt+ \sigma^n(t,X^n_t,\mathcal{L}_{X^n_t})dW_t\\
X^n_0 &= X_0 \\
\end{aligned}
\right.
\end{equation}
has a unique solution $X^n$ satisfying
\begin{equation}\label{ESofXn}
E\sup_{t\in [0,T]}|X^n_t|^2<\infty.
\end{equation}
Define
\begin{align*}
\tau^n:= \inf\{ t\geq 0: |X^n_t|\geq n \}.
\end{align*}
By the definition of $\phi_n$, we get
\begin{align*}
\phi_n(X ^n_t)= \frac{X^n_t \cdot n}{|X^n_t| \vee n}= X^n_{t\wedge \tau^n},
\end{align*}
and for every measurable set $A \subset \mathbb R^d$, we have
\begin{align*}
&{(\mathcal{L}_{X^n_t}\circ \phi_n^{-1}})(A)= P(X^n_t \in \phi_n^{-1}(A))\\
={}& P(\phi_n(X^n_t)\in A)=\mathcal{L}_{\phi_n (X^n_t)}(A)= \mathcal{L}_{X^n_{t\wedge \tau^n}}(A).
\end{align*}
Hence the equation \eqref{TMVSDE} is equivalent to the following SDE:
\begin{equation}\label{TEMVSDE}
\left\{
\begin{aligned}
dX^n_t&= b(t,X^n_{t\wedge \tau^n},\mathcal{L}_{X^n_{t\wedge \tau^n}})dt+ \sigma(t,X^n_{t\wedge \tau^n},\mathcal{L}_{X^n_{t\wedge \tau^n}})dW_t\\
X^n_0 &= X_0. \\
\end{aligned}
\right.
\end{equation}
Therefore, the stopped process $Y^n_t:=X^n_{t\wedge \tau^n}$ satisfies the following equation:
\begin{align}\label{SMVSDE}
\left\{
\begin{aligned}
dY^n_t
={} &
\mathbb I_{[0,\tau^n]}(t)b(t,Y^n_t,\mathcal{L}_{Y^n_t})dt
 +
\mathbb I_{[0,\tau^n]}(t)\sigma(t,Y^n_t,\mathcal{L}_{Y^n_t})dW_t\\
Y^n_0 = {}& X_0, \\
\end{aligned}
\right.
\end{align}
where $\mathbb I$ denotes the indicator function.

By inequality \eqref{ESofXn}, we have
\begin{align*}
E|Y^n_t|^2
&={}
E|X^n_{t\wedge \tau^n}|^2=\int_{\{t\le \tau^n\}}|X^n_t|^2dP+\int_{\{t> \tau^n\}}n^2 dP\\
&\le{}
E|X^n_t|^2+n^2<\infty,
\end{align*}
i.e. $\mathcal L_{Y^n_t}\in\mathcal P_2(\mathbb R^d)$ for any $n\ge 1$.
Thus, applying It\^o's formula to $V(Y^n_t,\mathcal{L}_{Y^n_t})$,
we obtain
\begin{align*}
&V(Y^n_t,\mathcal{L}_{Y^n_t})- V(X_0, \mathcal L_{X_0})\\
={} &
\int_0^t \partial_{x} V(Y^n_s,\mathcal{L}_{Y^n_s})\cdot \mathbb I_{[0,\tau^n]}(s)b(s,Y^n_s,\mathcal{L}_{Y^n_s})\\
& +{}
\frac12 \partial_{x} ^2V(Y^n_s,\mathcal{L}_{Y^n_s})\cdot \mathbb I_{[0,\tau^n]}(s)(\sigma\sigma^{\top})(s,Y^n_s,\mathcal{L}_{Y^n_s})\\
&+{}
\widetilde E\bigg[\partial_{\mu} V(Y^n_s,\mathcal{L}_{Y^n_s})(\widetilde {Y^n_s})\cdot \mathbb I_{[0,\tau^n]}(s)b(s,\widetilde {Y^n_s},\mathcal{L}_{Y^n_s})\\
& \quad+{}
\frac12 \partial_{y}\partial_{\mu}V(Y^n_s,\mathcal{L}_{Y^n_s})(\widetilde {Y^n_s}) \cdot \mathbb I_{[0,\tau^n]}(s)(\sigma\sigma^{\top})(s,\widetilde {Y^n_s},\mathcal{L}_{Y^n_s})\bigg]ds\\
& +
\int_0^t \langle \partial_x V(Y^n_s,\mathcal{L}_{Y^n_s}),  \mathbb I_{[0,\tau^n]}(s)\sigma(t,Y^n_s,\mathcal{L}_{Y^n_s})dW_s\rangle\\
={} &
\int_0^t \mathbb I_{[0,\tau^n]}(s)\mathcal LV(Y^n_s,\mathcal{L}_{Y^n_s})ds\\
& +
\int_0^t \langle \partial_x V(Y^n_s,\mathcal{L}_{Y^n_s}),  \mathbb I_{[0,\tau^n]}(s)\sigma(t,Y^n_s,\mathcal{L}_{Y^n_s})dW_s\rangle,
\end{align*}
where $\mathcal L$ denotes the generator.
Taking expectation on both sides, we have
\begin{align*}
&EV(Y^n_t,\mathcal{L}_{Y^n_t})- EV(X_0, \mathcal L_{X_0})
={}
E\int_0^t \mathbb I_{[0,\tau^n]}(s)\mathcal LV(Y^n_s,\mathcal{L}_{Y^n_s})ds\\
\le{} &
E\int_0^t \mathbb I_{[0,\tau^n]}(s)\lambda V(Y^n_s,\mathcal{L}_{Y^n_s})ds
\le{}
E\int_0^t \lambda V(Y^n_s,\mathcal{L}_{Y^n_s})ds,
\end{align*}
where the first inequality is due to (H2). By Gronwall's inequality, we have
\begin{align}\label{ESofT}
EV(Y^n_t,\mathcal{L}_{Y^n_t})
\le{}
EV(X_0, \mathcal L_{X_0})e^{\lambda t}.
\end{align}
On the other hand, we get
\begin{align*}
&EV(Y^n_t,\mathcal{L}_{Y^n_t})
={}EV(X^n_{t\wedge \tau^n},\mathcal{L}_{X^n_{t\wedge \tau^n}})\\
={} &
\int_{\{t< \tau^n\}}V(X^n_{t\wedge \tau^n},\mathcal{L}_{X^n_{t\wedge \tau^n}})dP
+\int_{\{t\ge \tau^n\}}V(X^n_{t\wedge \tau^n},\mathcal{L}_{X^n_{t\wedge \tau^n}})dP\\
\ge{} &
P(t\ge\tau^n)V(n,\mathcal{L}_{X^n_{t\wedge\tau^n}}).
\end{align*}
Thus, we obtain
\begin{equation}\label{Lim}
P(t\ge\tau^n) \le \frac{EV(Y^n_t,\mathcal{L}_{Y^n_t})}{\inf_{n\to\infty}V(n,\mathcal{L}_{X^n_{t\wedge\tau^n}})}
\le \frac{EV(X_0, \mathcal L_{X_0})e^{\lambda t}}{\inf_{n\to\infty}V(n,\mathcal{L}_{X^n_{t\wedge\tau^n}})}
\to 0, \quad \hbox{ as}~n\to\infty.
\end{equation}

By (H4), BDG's inequality, H\"older's inequality and inequality \eqref{ESofT}, there exists a constant $C(\ell,K)>0$ such that
\begin{align}\label{ESofUC}
& E(| Y^n_t- Y^n_s|^{2\ell})\\\nonumber
={} &
E\bigg(\bigg|\int_s^t \mathbb I_{[0,\tau^n]}(r)b(r,Y^n_r,\mathcal{L}_{Y^n_r})dr
+ \int_s^t \mathbb I_{[0,\tau^n]}(r) \sigma(r,Y^n_r,\mathcal{L}_{Y^n_r})dW_r \bigg|^{2\ell}\bigg)\\\nonumber
\leq{} &
C(\ell) E\int_s^t|\mathbb I_{[0,\tau^n]}(r)b(r,Y^n_r,\mathcal{L}_{Y^n_r})|^{2\ell}dr\cdot |t-s|^{2\ell-1}\\\nonumber
& +{}
C(\ell)E\bigg(\int_s^t|\mathbb I_{[0,\tau^n]}(r)\sigma(r,Y^n_r,\mathcal{L}_{Y^n_r})|^2 dr\bigg)^{\ell}\\\nonumber
\leq{} &
C(\ell,K) E\int_s^t1+V(Y^n_r,\mathcal{L}_{Y^n_r})dr\cdot |t-s|^{2\ell-1}\\\nonumber
& +{}
C(\ell)E\bigg(\int_s^t|\mathbb I_{[0,\tau^n]}(r)\sigma(r,Y^n_r,\mathcal{L}_{Y^n_r})|^{2\ell} dr\bigg)\cdot |t-s|^{\ell-1}\\\nonumber
\le{} &
C(\ell,K)\big(1+e^{\lambda T}EV(X_0,\mathcal L_{X_0})\big)\cdot |t-s|^{\ell}.
\end{align}
Let $k=\left[\frac{T}{\epsilon}\right]+1$ where $[a]$ denotes the integer part of $a\in \mathbb R$. Then we get
\begin{align*}
&E\big(\sup_{s,t\in[0,T],|t-s|\leq \epsilon}| Y^n_t-Y^n_s|^{2\ell}\big)\\
\leq{} &
C(\ell,k)\sum_{j=1}^{k} E\big(| Y^n_{j\epsilon}-Y^n_{(j-1)\epsilon}|^{2\ell}\big)\\
\leq{} &
C(\ell,k,K)\big(1+e^{\lambda T}EV(X_0,\mathcal L_{X_0})\big)\cdot(T+\epsilon)\epsilon^{\ell-1}.
\end{align*}
Thus, by Ascoli-Arzela theorem, $\{\mathcal L_{Y^n}\}$ is tight on $C[0,T]$,
i.e. there exists a subsequence, still denoted $\{\mathcal L_{Y^n}\}$, which is weakly convergent to some measure $\mu\in\mathcal P(C[0,T])$.
By Skorokhod's representation theorem, there are $C[0,T]$-valued random variables $\overline {Y^n}$ and $\bar X$
on some probability space $(\bar \Omega,\bar {\mathcal F}, \bar P)$
such that $\mathcal L_{\overline {Y^n}}=\mathcal L_{Y^n}, \mathcal L_{\bar X}=\mu$, and
\begin{align}\label{Limit}
\overline {Y^n}\to \bar X~a.s.
\end{align}
Moreover, by H\"older's inequality, BDG's inequality, (H4) and inequality \eqref{ESofT} we obtain
\begin{align*}
\bar E\sup_{t\in [0,T]}| \overline {Y^n_t}|^{2\ell}
={} &
E\sup_{t\in [0,T]}| Y^n_t|^{2\ell}\\
\le{} &
C(\ell)E(\sup_{t\in [0,T]}| Y^n_t- Y^n_0|^{2\ell})+ C(\ell)E|Y^n_0|^{2\ell}\\
\le{} &
C(\ell)E\bigg(\sup_{t\in [0,T]}\bigg|\int_0^t \mathbb I_{[0,\tau^n]}(r)b(r,Y^n_r,\mathcal{L}_{Y^n_r})dr\\
& +{} \int_0^t \mathbb I_{[0,\tau^n]}(r) \sigma(r,Y^n_r,\mathcal{L}_{Y^n_r})dW_r \bigg|^{2\ell}\bigg)+ C(\ell)E|Y^n_0|^{2\ell}\\\nonumber
\leq{} &
C(\ell) E\int_0^T|\mathbb I_{[0,\tau^n]}(r)b(r,Y^n_r,\mathcal{L}_{Y^n_r})|^{2\ell}dr\cdot T^{2\ell-1}\\\nonumber
& +{}
C(\ell)E\bigg(\int_0^T|\mathbb I_{[0,\tau^n]}(r)\sigma(r,Y^n_r,\mathcal{L}_{Y^n_r})|^2 dr\bigg)^{\ell}+ C(\ell)E|Y^n_0|^{2\ell}\\\nonumber
\leq{} &
C(\ell,K) E\int_0^T1+V(Y^n_r,\mathcal{L}_{Y^n_r})dr\cdot T^{2\ell-1}\\\nonumber
& +{}
C(\ell)E\bigg(\int_0^T|\mathbb I_{[0,\tau^n]}(r)\sigma(r,Y^n_r,\mathcal{L}_{Y^n_r})|^{2\ell} dr\bigg)\cdot T^{\ell-1}+ C(\ell)E|Y^n_0|^{2\ell}\\\nonumber
\le{} &
C(\ell,K)\big(1+e^{\lambda T}EV(X_0,\mathcal L_{X_0})\big)\cdot T^{\ell}+ C(\ell)E|Y^n_0|^{2\ell}.
\end{align*}
where $\bar E$ denotes the expectation on the probability space $(\bar \Omega,\bar {\mathcal F}, \bar P)$,
and the first equality holds by $\mathcal L_{\overline {Y^n}}=\mathcal L_{Y^n}$.
Thus, by Vitali convergence theorem we have
\begin{equation}\label{LofC}
\lim_{n\to\infty}\bar E\sup_{0\le t\le T}|\overline {Y^n_t}-\bar X_t|^{2\ell}=0.
\end{equation}
From this, we get that $\mathcal L_{\overline {Y^n}}$ converges weakly to $\mathcal L_{\bar X}$ in $\mathcal P_{2\ell}(C[0,T])$.
Furthermore, due to the continuity of Lyapunov function $V$, we obtain
$$
V(\overline {Y^n_t},\mathcal L_{\overline {Y^n_t}})\to V(\bar X_t,\mathcal{L}_{\bar X_t})\quad a.s.
$$
Applying inequality \eqref{ESofT} and Vitali convergence theorem, we have
\begin{equation}\label{ESofWS}
EV(\bar X_t,\mathcal{L}_{\bar X_t})
\le
EV(X_0, \mathcal L_{X_0})e^{\lambda t}.
\end{equation}

Claim:
$\overline {M^n_t}:= \overline {Y^n_t}- \overline {Y^n_0}- \int_0^t \mathbb I_{[0,\tau^n]}(r)b(r,\overline {Y^n_r},\mathcal{L}_{\overline {Y^n_r}})dr$ is a square-integrable martingale, and its quadratic variation
\begin{equation*}
\left[\overline {M^n}\right]_t= \int_0^t \mathbb I_{[0,\tau^n]}(r)(\sigma\cdot \sigma^T)(r,\overline {Y^n_r},\mathcal{L}_{\overline {Y^n_r}})dr.
\end{equation*}
Indeed, by the identical distribution between $\overline {Y^n}$ and $Y^n$, we have
\begin{align}\label{Mar1}
&\bar E[(\overline {M^n_t}-\overline {M^n_s})\cdot \Phi(\overline {Y^n}|_{[0,s]})]\\\nonumber
= {}&
\bar E\bigg[\bigg(\overline {Y^n_t}- \overline {Y^n_s}- \int_s^t \mathbb I_{[0,\tau^n]}(r)b(r,\overline {Y^n_r},\mathcal{L}_{\overline {Y^n_r}})dr\bigg)\cdot \Phi(\overline {Y^n}|_{[0,s]})\bigg]\\\nonumber
= {}&
0,
\end{align}
and
\begin{align}\label{Mar2}
&\bar E\bigg[\bigg(\overline {M^n_t}\cdot\overline {M^n_t}-\overline {M^n_s}\cdot\overline {M^n_s}\\\nonumber
&\quad - \int_s^t \mathbb I_{[0,\tau^n]}(r)(\sigma\cdot \sigma^T)(r,\overline {Y^n_r},\mathcal{L}_{\overline {Y^n_r}})dr)\cdot \Phi(\overline {Y^n}|_{[0,s]}\bigg)\bigg]
=0,
\end{align}
for any $\Phi\in C[0,T]$.

By the limit \eqref{Limit} and the continuity of coefficient $b$, we get
\begin{align*}
\lim_{n\to\infty}|b(r,\overline {Y^n_r},\mathcal{L}_{\overline {Y^n_r}})-b(r,\bar X_{r},\mathcal{L}_{\bar X_{r}})|=0
\quad \bar P\raisebox{0mm}{-} a.s.
\end{align*}
and due to (H4), inequalities \eqref{ESofT} and \eqref{ESofWS}, we have
\begin{align*}
&\bar E|b(r,\overline {Y^n_r},\mathcal{L}_{\overline {Y^n_r}})-b(r,\bar X_{r},\mathcal{L}_{\bar X_{r}})|^{2\ell}\\
\le{} &
C(\ell)E|b(r,\overline {Y^n_r},\mathcal{L}_{\overline {Y^n_r}})|^{2\ell}
+ C(\ell)E|b(r,\bar X_{r},\mathcal{L}_{\bar X_{r}})|^{2\ell}\\
\le{} &
C(\ell)(1+ EV(\overline {Y^n_r},\mathcal{L}_{\overline {Y^n_r}})+EV(\bar X_{r},\mathcal{L}_{\bar X_{r}}))\\
< {}&\infty.
\end{align*}
Therefore, by Vitali convergence theorem, we obtain
\begin{align*}
\lim_{n\to\infty}\bar E|b(r,\overline {Y^n_r},\mathcal{L}_{\overline {Y^n_r}})-b(r,\bar X_{r},\mathcal{L}_{\bar X_{r}})|^{2\ell}=0.
\end{align*}
In a similar way, we have
\begin{align*}
\lim_{n\to\infty}\bar E|\sigma(r,\overline {Y^n_r},\mathcal{L}_{\overline {Y^n_r}})-\sigma(r,\bar X_{r},\mathcal{L}_{\bar X_{r}})|^{2\ell}=0.
\end{align*}
Thus, let $n\to\infty$ in equality \eqref{Mar1} as well as \eqref{Mar2} and denote $\bar M:=\lim_{n\to\infty}\overline {M^n}$,
by the limit \eqref{Lim} we get
\begin{align*}
&\bar E[(\bar M_t-\bar M_s)\cdot \Phi(\bar X|_{[0,s]})]=0,\\
&\bar E\big[\big(\bar M_t\cdot\bar M_t-\bar M_s\cdot\bar M_s- \int_s^t\sigma\cdot \sigma^T(r,\bar X_{r},\mathcal{L}_{\bar X_{r}})dr\big)\cdot \Phi(\bar X|_{[0,s]})\big]
=0.
\end{align*}
By the martingale limit theorem, see for instance \cite[Proposition 1.3]{Chung-Williams},
$\bar M(t)$ is a square integrable martingale, and its quadratic variation
$$
\left[M\right]_t= \int_0^t \sigma\cdot \sigma^T(r,\bar X_r,\mathcal{L}_{\bar X_r})dr.
$$
Therefore, by the martingale representation theorem, there exists a Brownian motion $\bar W$
on the probability space $(\bar \Omega,\bar {\mathcal F}, \bar P)$ such that
$$
M(t)= \int_0^t \sigma(r,\bar X_r,\mathcal{L}_{\bar X_r})d\bar W_r.
$$
Moreover, due to equation \eqref{SMVSDE} and $\mathcal L_{\overline {Y^n}}=\mathcal L_{Y^n}$, we have
\begin{align}\label{SMVSDE1}
\left\{
\begin{aligned}
d\overline {Y^n_t}
={} &
\mathbb I_{[0,\tau^n]}(t)b(t,\overline {Y^n_t},\mathcal{L}_{\overline {Y^n_t}})dt
 +
d\overline {M^n_t}\\
\overline {Y^n_0} = {}& X_0. \\
\end{aligned}
\right.
\end{align}
Thus, let $n\to\infty$, we obtain that $(\bar X, \bar W)$ is a weak solution for MVSDE \eqref{MVSDE}.

Next, we prove the strong existence for equation \eqref{MVSDE}.
Let $\mu_t=\mathcal L_{\bar X_t}$, there has strong uniqueness under the assumptions (H1) and (H2)
for the following SDE:
\begin{align*}
dX_t= b(t,X_t,\mu_t)dt+ \sigma(t,X_t,\mu_t)dW_t.
\end{align*}
Thus, by a modified Yamada-Watanabe theorem (see \cite[Lemma 2.1]{HW}), the equation \eqref{MVSDE} has a strong solution $X$.

We could also obtain the weak uniqueness by the modified Yamada-Watanabe theorem.
That is to say that $\mathcal L_X=\mathcal L_{\bar X}$.
Therefore the estimate \eqref{ES} follows from inequality \eqref{ESofWS}.
The proof is complete.
\end{proof}
\begin{rem}\label{Re1}
\begin{enumerate}
\item By taking $V(x,\mu)=(x+\int y\mu(dy))^{\kappa\ell}$ where $\kappa>2$, (H4) reduces to the growth condition $A4$ in \cite{HHL}.
      And if $V$ is independent of $\mu$, (H4) is equivalent to the Lyapunov condition in \cite{BRS}.

\item The Lyapunov condition in \cite{LM} implies that the dependence on distribution variable of coefficients is `linear'.
That is to say that the dependence is of the form $\mu(f)$ rather than the form $(\mu(f))^p$
where $f:\mathbb R^d\to \mathbb R^d$ is a function, and constant $p>1$.
However, here we do not have this restriction.

\item Note that the distribution $\mathcal L_{\bar X}$ denotes the push forward measure
      on the probability space $(\bar \Omega,\bar {\mathcal F}, \bar P)$.
      While a more precise notation would be $\mathcal L_{\bar X|\bar P}$,
      we think that it does not cause any confusion, so we abbreviate it as $\mathcal L_{\bar X}$ in the context.

\item By the estimate \eqref{ES} and condition (H4), we get $\mathcal L_{X_t}\in \mathcal P_2(\mathbb R^d)$
which indicates that It\^o's formula could be applied to function $f\in C^{2,(1,1)}(\mathbb R^d\times \mathcal P_2(\mathbb R^d))$.
\end{enumerate}
\end{rem}

\section{Ergodicity}
In this section, we study ergodicity for MVSDEs under Lyapunov conditions.
Before giving the main results, we define the semigroups $P_t$ and $P_t^*$,
and get linearity of the corresponding Fokker-Planck equations on $\mathbb R^d\times \mathcal P_2(\mathbb R^d)$,
which implies that the Krylov-Bogolioubov theorem is valid.
The main result in this section is ergodic in the weak topology sense, which is obtained by the Krylov-Bogolioubov theorem,
under the distribution-dependent Lyapunov condition.
Moreover, if the Lyapunov function depends only on space variable, we obtain exponential ergodicity for semigroup $P_t^*$ on $\mathcal P(\mathbb R^d\times \mathcal P_2(\mathbb R^d))$ under Wasserstein quasi-distance induced by the Lyapunov function.

Before illustrating the main results in this section, we make some preparations about semigroups $P_t$ on $C_b(\mathbb R^d\times \mathcal P_2(\mathbb R^d))$ and $P_t^*$ on $\mathcal P(\mathbb R^d\times \mathcal P_2(\mathbb R^d))$.
We consider the following coupled SDEs:
\begin{align}
\left\{
\begin{aligned}\label{MVSDE1}
&dX_t= b(t,X_t,\mathcal{L}_{X_t})dt+ \sigma(t,X_t,\mathcal{L}_{X_t})dW_t,~\mathcal L_{X_0}=\mu,\\
&d\bar X_t= \bar b(t,\bar X_t,\mathcal{L}_{X_t})dt+ \bar \sigma(t,\bar X_t,\mathcal{L}_{X_t})dW_t,~\bar X_0=x,
\end{aligned}
\right.
\end{align}
where $x\in\mathbb R^d,\mu\in\mathcal P_2(\mathbb R^d)$.
If the coefficients $b,\sigma,\bar b,\bar \sigma$ satisfy some conditions to be specified below, we have
\begin{align}\label{ItoCouple}
V(\bar X_t,\mathcal L_{X_t})
={}&
V(\bar X_0,\mathcal L_{X_0})+ \int_0^t\bar b(s,\bar X_s,\mathcal L_{X_s})\cdot \partial_x V(\bar X_s,\mathcal L_{X_s})\\\nonumber
&+{} \frac 12tr((\bar\sigma\bar\sigma^{\top})(s,\bar X_s,\mathcal L_{X_s})\partial_x^2 V(\bar X_s,\mathcal L_{X_s}))ds\\\nonumber
&+{} \int_0^t\bar\sigma(s,\bar X_s,\mathcal L_{X_s})\cdot \partial_x V(\bar X_s,\mathcal L_{X_s})dW_s\\\nonumber
&+{} \tilde E\int_0^t\bigg( b(s,\tilde X_s,\mathcal L_{X_s})\partial_{\mu} V(\bar X_s,\mathcal L_{X_s})(\tilde X_s)\\\nonumber
&\quad+ \frac 12tr((\sigma\sigma^{\top})(s,\tilde X_s,\mathcal L_{X_s})\partial_y \partial_{\mu} V(\bar X_s,\mathcal L_{X_s})(\tilde X_s))\bigg)ds\\\nonumber
=:{}&
V(\bar X_0,\mathcal L_{X_0})+ \int_0^t \bar L_1V(\bar X_s,\mathcal L_{X_s})+L_2V(\bar X_s,\mathcal L_{X_s})ds\\\nonumber
&+{}\int_0^t\bar\sigma(s,\bar X_s,\mathcal L_{X_s})\cdot \partial_x V(\bar X_s,\mathcal L_{X_s})dW_s
\end{align}
for $V\in C^{2,(1,1)}(\mathbb R^d\times \mathcal P_2(\mathbb R^d))$,
where
\begin{align*}
\bar L_1V(\bar X_s,\mathcal L_{X_s})
:={}&
\bar b(s,\bar X_s,\mathcal L_{X_s})\cdot \partial_x V(\bar X_s,\mathcal L_{X_s})+{} \frac 12tr((\bar\sigma\bar\sigma^{\top})(s,\bar X_s,\mathcal L_{X_s})\partial_x^2 V(\bar X_s,\mathcal L_{X_s})),
\end{align*}
and
\begin{align*}
L_2V(\bar X_s,\mathcal L_{X_s})
:={}&
\tilde E\bigg[ b(s,\tilde X_s,\mathcal L_{X_s})\partial_{\mu} V(\bar X_s,\mathcal L_{X_s})(\tilde X_s)\\
&+{} \frac 12tr((\sigma\sigma^{\top})(s,\tilde X_s,\mathcal L_{X_s})\partial_y \partial_{\mu} V(\bar X_s,\mathcal L_{X_s})(\tilde X_s))\bigg].
\end{align*}
Define
$$
P_tf(x,\mu):=\int_{\mathbb R^d\times \mathcal P_2(\mathbb R^d)}f(y,\nu)\mathcal L_{\bar X_t}(dy)\times \delta_{\{\mathcal L_{X_t}\}}(d\nu)
$$
for $f\in \mathcal B_b(\mathbb R^d\times \mathcal P_2(\mathbb R^d))$.
Thus, we have
\begin{align*}
P_tf(x,\mu)
=
\int_{\mathbb R^d}f(y,\mathcal L_{X_t})\mathcal L_{\bar X_t}(dy)
=
Ef(\bar X_t,\mathcal L_{X_t}).
\end{align*}
Moreover, if $f\in C^{2,(1,1)}(\mathbb R^d\times \mathcal P_2(\mathbb R^d))$, we get
$$
P_tf(x,\mu)
=
Ef(\bar X_t,\mathcal L_{X_t})
=
Ef(\bar X_0,\mathcal L_{X_0})+ E\int_0^t \bar L_1f(\bar X_s,\mathcal L_{X_s})+L_2f(\bar X_s,\mathcal L_{X_s})ds.
$$
Let $\pi_t:=\mathcal L_{\bar X_t}\times \delta_{\{\mathcal L_{X_t}\}}$ and $L_{\bar 1,2}f:=\bar L_1f+L_2f$, we obtain
\begin{align*}
\int fd\pi_t
&={}\int fd\pi_0+ \int_0^t\int L_{\bar 1,2}fd\pi_sds,
\end{align*}
i.e. $\partial \pi_t=L_{\bar 1,2}^*\pi_t$ is a linear equation. Here $L_{\bar 1,2}^*$ denotes the adjoint operator of $L_{\bar 1,2}$.
In particular, $L_{\bar 1,2}f=\mathcal Lf$ if $b=\bar b,\sigma=\bar\sigma$ and $\mu=\delta_x$.

Next we apply linearity of the corresponding Fokker-Planck equation to obtain that
\begin{align*}
P_t(P_sf)(x,\mu)
& ={}
\int_{\mathbb R^d\times \mathcal P_2(\mathbb R^d)}P_sf(y,\nu)\mathcal L_{\bar X_t^x}(dy)\times \delta_{\{\mathcal L_{X_t^{\mu}}\}}(d\nu)\\
& ={}
\int_{\mathbb R^d\times \mathcal P_2(\mathbb R^d)}\bigg(\int_{\mathbb R^d\times \mathcal P_2(\mathbb R^d)}f(z, m)\mathcal L_{\bar X_s^y}(dz)\times \delta_{\{\mathcal L_{X_t^{\nu}}\}}(dm)\bigg)\cdot\mathcal L_{\bar X_t^x}(dy)\times \delta_{\{\mathcal L_{X_t^{\mu}}\}}(d\nu)\\
& ={}
\int_{\mathbb R^d\times \mathcal P_2(\mathbb R^d)}f(z, m)\int_{\mathbb R^d\times \mathcal P_2(\mathbb R^d)}\mathcal L_{\bar X_s^y}(dz)\times \delta_{\{\mathcal L_{X_t^{\nu}}\}}(dm)\cdot\mathcal L_{\bar X_t^x}(dy)\times \delta_{\{\mathcal L_{X_t^{\mu}}\}}(d\nu)\\
& ={}
\int_{\mathbb R^d\times \mathcal P_2(\mathbb R^d)}f(z, m)\mathcal L_{\bar X_{t+s}^{x}}(dz)\times \delta_{\{\mathcal L_{X_{t+s}^\mu}\}}(dm)\\
& ={}
(P_{t+s}f)(x,\mu),
\end{align*}
where the third equality holds by Fubini theorem.
This means that the operator $P_t$ is a semigroup: $P_t\circ P_s= P_{t+s}$.

Define another operator semigroup $P_t^*, t\ge 0$ acting on the space $\mathcal P(\mathbb R^d\times \mathcal P_2(\mathbb R^d))$ with
$$
P_t^*\pi(dy,d\nu):= \int_{\mathbb R^d\times \mathcal P_2(\mathbb R^d)}\mathcal L_{\bar X_t^x}(dy)\times \delta_{\{\mathcal L_{X_t^{\mu}}\}}(d\nu)\pi(dx,d\mu),
$$
for any $\pi\in \mathcal P(\mathbb R^d\times \mathcal P_2(\mathbb R^d))$.
Moreover, the operators $P_t$ and $P_t^*$ are conjugate to each other in the sense that
\begin{align*}
&\int_{\mathbb R^d\times \mathcal P_2(\mathbb R^d)}P_tf(x,\mu)\pi(dx,d\mu)\\
={} &
\int_{\mathbb R^d\times \mathcal P_2(\mathbb R^d)}\bigg(\int_{\mathbb R^d\times \mathcal P_2(\mathbb R^d)}f(y,r)\mathcal L_{\bar X_t^x}(dy)\times \delta_{\{\mathcal L_{X_t^{\mu}}\}}(dr)\bigg)\cdot \pi(dx,d\mu)\\
={} &
\int_{\mathbb R^d\times \mathcal P_2(\mathbb R^d)}f(y,r)\int_{\mathbb R^d\times \mathcal P(\mathbb R^d)}\mathcal L_{\bar X_t^x}(dy)\times \delta_{\{\mathcal L_{X_t^{\mu}}\}}(dr)\cdot \pi(dx,d\mu)\\
={} &
\int_{\mathbb R^d\times \mathcal P_2(\mathbb R^d)}f(y,r)\mathcal P_t^*\pi(dy,dr),
\end{align*}
where the second equality holds by Fubini theorem.

Now we show that semigroup $P_t$ is Feller.
\begin{lem}\label{Lem1}
Assume that there exist constants $K,L>0$ such that the coefficients $b,\bar b,\sigma,\bar\sigma$ are Lipschitz with Lipschitz constant $L$,
and satisfy the linear growth conditions with linear growth constant $K$.
Then the semigroup $P_t$ is Feller.
\end{lem}
\begin{proof}
Note that for any sequence $\{\mu_n\}\subset\mathcal P_2(\mathbb R^d)$ with $\lim_{n\to\infty}W_2(\mu_n,\mu)=0$,
by \cite[Theorem 4.1]{LM} there exist random variables $X^{\mu_n}_0,X^{\mu}_0$
such that $\mathcal L_{X^{\mu_n}_0}=\mu_n,\mathcal L_{X^{\mu}_0}=\mu$ and
\begin{equation}\label{W2}
W_2(\mu_n,\mu)^2=E|X^{\mu_n}_0-X^{\mu}_0|^2.
\end{equation}
Let $X^{\mu_n}_t,X^{\mu}_t$ denote the solutions of the following SDEs with initial values $X^{\mu_n}_0,X^{\mu}_0$ respectively,
$$dX_t= b(t,X_t,\mathcal{L}_{X_t})dt+ \sigma(t,X_t,\mathcal{L}_{X_t})dW_t.$$
Note that we have
\begin{align*}
E|X^{\mu_n}_t-X^{\mu}_t|^2
=&{}
E|X^{\mu_n}_0-X^{\mu}_0|^2
+
2E\int_0^t\langle X^{\mu_n}_s-X^{\mu}_s,b(s,X^{\mu_n}_s,\mathcal{L}_{X^{\mu_n}_s})-b(s,X^{\mu}_s,\mathcal{L}_{X^{\mu}_s})\rangle ds\\
&+
E\int_0^t|\sigma(s,X^{\mu_n}_s,\mathcal{L}_{X^{\mu_n}_s})-\sigma(s,X^{\mu}_s,\mathcal{L}_{X^{\mu}_s})|^2ds\\
\le&{}
E|X^{\mu_n}_0-X^{\mu}_0|^2
+
C(L)\int_0^tE|X^{\mu_n}_s-X^{\mu}_s|^2ds.
\end{align*}
By Gronwall's inequality, we get
\begin{align}\label{Lem4.2}
W_2(\mathcal{L}_{X^{\mu_n}_t},\mathcal{L}_{X^{\mu}_t})
&\le{}
E|X^{\mu_n}_t-X^{\mu}_t|^2
\le{}
e^{C(L)\cdot t}E|X^{\mu_n}_0-X^{\mu}_0|^2
={}e^{C(L)\cdot t}W_2(\mu_n,\mu).
\end{align}

For any sequence $\{x_n\}\subset\mathbb R^d$ with $x_n\to x$ as $n\to\infty$, let $\bar X^{x_n}_t,\bar X^{x}_t$ denote the solutions of the following SDEs respectively,
\begin{align*}
d\bar X_t^{x_n} &={} \bar b(t,\bar X_t^{x_n},\mathcal{L}_{X^{\mu_n}_t})dt+ \bar\sigma(t,\bar X_t^{x_n},\mathcal{L}_{X^{\mu_n}_t})dW_t,~\bar X_0^{x_n}=x_n\\
d\bar X_t^x &={} \bar b(t,\bar X_t^x,\mathcal{L}_{X^{\mu}_t})dt+ \bar\sigma(t,\bar X_t^x,\mathcal{L}_{X^{\mu}_t})dW_t,~\bar X_t^x=x.
\end{align*}
In a similar way, we have
\begin{align*}
E|\bar X_t^{x_n}-\bar X_t^x|^2
=&{}
E|\bar X_0^{x_n}-\bar X_0^x|^2
+
2E\int_0^t\langle \bar X^{x_n}_s-\bar X^x_s,b(s,\bar X^{x_n}_s,\mathcal{L}_{X^{\mu_n}_s})-b(s,\bar X^x_s,\mathcal{L}_{X^{\mu}_s})\rangle ds\\
&+
E\int_0^t|\sigma(s,\bar X^{x_n}_s,\mathcal{L}_{X^{\mu_n}_s})-\sigma(s,\bar X^x_s,\mathcal{L}_{X^{\mu}_s})|^2ds\\
\le&{}
E|\bar X_0^{x_n}-\bar X_0^x|^2
+
C(L)\int_0^tE|X^{\mu_n}_s-X^{\mu}_s|^2+ W_2(\mathcal{L}_{X^{\mu_n}_s},\mathcal{L}_{X^{\mu}_s})^2ds.
\end{align*}
Applying Gronwall's inequality and inequality \eqref{Lem4.2}, we get
\begin{align*}
E|\bar X_t^{x_n}-\bar X_t^x|^2
\le&{}
\bigg(E|\bar X_0^{x_n}-\bar X_0^x|^2+C(L)\int_0^tW_2(\mathcal{L}_{X^{\mu_n}_s},\mathcal{L}_{X^{\mu}_s})^2ds\bigg)\cdot e^{C(L)t}\\
\le&{}
\bigg(|x_n-x|+C(L)(e^{C(L)t}-1)W_2(\mu_n,\mu)^2\bigg)\cdot e^{C(L)t}.
\end{align*}

Therefore, let $n\to\infty$ in the above inequality, it follows that $\bar X_t^{x_n}$ converges to $\bar X_t^x$ in $L^2$.
By Skorokhod's representation theorem, there are two random variables $\tilde X_t^{x_n}$ and $\tilde X_t^x$
on some probability space $(\tilde \Omega, \tilde {\mathcal F},\tilde P)$
such that $\mathcal L_{\bar X_t^{x_n}}=\mathcal L_{\tilde X_t^{x_n}}$, $\mathcal L_{\bar X_t^{x}}=\mathcal L_{\tilde X_t^{x}}$, and
$$\tilde X_t^{x_n}\to \tilde X_t^x\quad a.s.$$
Thus, by Lebesgue dominated convergence, we have
\begin{align*}
P_tf(x_n,\mu_n)-P_tf(x,\mu)
=&{}
Ef(\bar X_t^{x_n},\mathcal L_{X^{\mu_n}_t})-Ef(\bar X_t^{x},\mathcal L_{X^{\mu}_t})\\
=&{}
Ef(\tilde X_t^{x_n},\mathcal L_{X^{\mu_n}_t})-Ef(\tilde X_t^{x},\mathcal L_{X^{\mu}_t})
\to 0 \quad \hbox{ as} \quad n\to\infty,
\end{align*}
for any $f\in C_b(\mathbb R^d\times\mathcal P_2(\mathbb R^d))$.
The proof is complete.
\end{proof}

Next we divide this section into two parts: ergodicity in the weak topology sense and exponential ergodicity under Wasserstein quasi-distance. Assume that the coefficients are independent of $t$ in the remaining parts of this section.

\subsection{Ergodicity in the weak topology sense}
In this subsection, we consider a special case in equations \eqref{MVSDE1}:
\begin{align}\label{Main4.1}
dX_t= b(X_t,\mathcal{L}_{X_t})dt+ \sigma(X_t,\mathcal{L}_{X_t})dW_t.
\end{align}
\begin{enumerate}
\item [(H6)](Lyapunov condition)
There exists a nonnegative function $V\in C^{2,(1,1)}(\mathbb R^d \times \mathcal P_2(\mathbb R^d))$ such that there is a constant $\gamma\ge 0$ satisfying
\begin{align*}
\mathcal LV(x,\mu)  &\leq{} -\gamma, \\
V_R&:={} \inf_{|x|\vee|\mu|_2\geq R}V(x, \mu)\to \infty ~as~ R\to \infty,
\end{align*}
for all $(x,\mu)\in \mathbb R^d \times \mathcal P_2(\mathbb R^d)$,
where $|\mu|_2^2:=\int_{\mathbb R^d}|x|^2\mu(dx)$.
\end{enumerate}
\begin{rem}
Since the condition (H6) is stronger than (H2), there exists a unique solution
with a finite second-order moment under conditions (H1), (H3)--(H6).
\end{rem}

\begin{thm}\label{Thm2}
Assume that (H1), (H3)-(H6) holds, semigroup $P_t$ is Feller,
and that there is a function $C:\mathbb R^+\to\mathbb R^+$ satisfying $\sup_{t\in\mathbb R^+}C(t)<\infty$ such that
\begin{align*}
\omega(\tilde P_t^*\mu,\tilde P_t^*\nu)
\le{}
C(t)\omega(\mu,\nu) \quad \hbox{for }\mu,\nu\in\mathcal P_2(\mathbb R^d),
\end{align*}
where $\tilde P_t^*\mu$ denotes the distribution at time $t$ with initial distribution $\mu$.
Then there exists at least one invariant measure.
Moreover, if there exists a constant $t_0>0$ such that $C(t_0)<1$,
there is a unique invariant measure.
\end{thm}
\begin{proof}
By It\^o's formula and (H6), we have
\begin{align*}
EV(X_t, \mathcal L_{X_t})
& ={}
EV(X_0, \mathcal L_{X_0})+ E\int_0^t \mathcal LV(X_s, \mathcal L_{X_s})ds\\
& \le{}
EV(X_0, \mathcal L_{X_0})- \gamma t.
\end{align*}
Denote $\pi_t:= \mathcal L_{X_t}\times \delta_{\{\mathcal L_{X_t}\}}$. Note that we have
\begin{align*}
EV(X_t, \mathcal L_{X_t})
={} &
\int_{\mathbb R^d\times \mathcal P(\mathbb R^d)}V(x,\mu)\mathcal L_{X_t}(dx)\times \delta_{\{\mathcal L_{X_t}\}}(d\mu)\\
\ge{} &
\int_{\{(x,\mu): V(x,\mu)>N\}}V(x,\mu)\pi_t(dx,d\mu)\\
\ge{} &
N\cdot\pi_t\{(x,\mu): V(x,\mu)>N\}.
\end{align*}
Moreover, by (H6) there is a constant $R>0$ such that
$$\{(x,\mu): V(x,\mu)\le N\}=\{(x,\mu): |x|\vee|\mu|_2\le R\},$$
which means that $\{(x,\mu): V(x,\mu)\le N\}$ is a compact set in $\mathbb R^d\times \mathcal P_2(\mathbb R^d)$.
Thus, $\{\pi_t\}$ is tight.
Then by the Krylov-Bogolioubov theorem, there exists an invariant measure $\pi$ for $(P_t^*)_{t\ge 0}$.
Moreover, $\pi(\mathcal P_2(\mathbb R^d))(\cdot):=\pi(\cdot\times\mathcal P_2(\mathbb R^d))$ is an invariant measure for equation \eqref{Main4.1}.

Indeed, we have
\begin{align*}
\tilde P_s^*\mathcal L_{X_t}=\mathcal L_{X_{t+s}}= (P_s^*\pi_t)(\mathcal P_2(\mathbb R^d))= \pi_{t+s}(\mathcal P_2(\mathbb R^d))
\end{align*}
and
\begin{align*}
\int_{\mathbb R^d}f(x)d\pi_{t}(x, \mathcal P_2(\mathbb R^d))
& =
\int_{\mathbb R^d\times \mathcal P_2(\mathbb R^d)}f(x)\mathbb I_{\mathcal P_2(\mathbb R^d)}(\mu)\pi_{t}(dx,d\mu)\\
& \to
\int_{\mathbb R^d\times \mathcal P_2(\mathbb R^d)}f(x)\mathbb I_{\mathcal P_2(\mathbb R^d)}(\mu)\pi(dx,d\mu)\\
& =
\int_{\mathbb R^d}f(x)\pi(dx,\mathcal P_2(\mathbb R^d)), ~~as ~t\to \infty
\end{align*}
for any $f\in C_b(\mathbb R^d)$;
i.e. $\pi_{t}(\mathcal P_2(\mathbb R^d))$ converges weakly to $\pi(\mathcal P_2(\mathbb R^d))$ as $t\to\infty$.
Thus, we get for $s\ge 0$,
\begin{align*}
&\omega(\tilde P_s^*\pi(\mathcal P_2(\mathbb R^d)),\pi(\mathcal P_2(\mathbb R^d)))\\
\le{} &
\lim_{t\to\infty}\omega(\tilde P_s^*\pi(\mathcal P_2(\mathbb R^d)),\tilde P_s^*\pi_t(\mathcal P_2(\mathbb R^d)))
 +
\lim_{t\to\infty}\omega(\tilde P_s^*\pi_t(\mathcal P_2(\mathbb R^d)),\pi_t(\mathcal P_2(\mathbb R^d)))\\
& +
\lim_{t\to\infty}\omega(\pi_t(\mathcal P_2(\mathbb R^d)),\pi(\mathcal P_2(\mathbb R^d)))\\
\le{} &
\lim_{t\to\infty}C(s)\omega(\pi(\mathcal P_2(\mathbb R^d)),\pi_{t}(\mathcal P_2(\mathbb R^d)))
+\lim_{t\to\infty}\omega(\pi_{t+s}(\mathcal P_2(\mathbb R^d)),\pi_t(\mathcal P_2(\mathbb R^d)))\\
={}&0.
\end{align*}
That is, $\pi(\mathcal P_2(\mathbb R^d))$ is an invariant measure.

If there exist two invariant measures $\mu,\nu$, we have
\begin{align*}
\omega(\mu,\nu)
={}
\omega(\tilde P_{t_0}^*\mu,\tilde P_{t_0}^*\nu)
\le{}
C(t_0)\omega(\mu,\nu).
\end{align*}
This ensures that $\mu=\nu$.
The proof is complete.
\end{proof}
\begin{rem}
\begin{enumerate}
\item The conditions `$\sup_{t\in\mathbb R^+}C(t)<\infty$' and `$C(t_0)<1$ for some $t_0>0$'
imply that there exists a contraction for $\tilde P_t^*$
as long as $t$ is sufficiently large.

\item The Krylov-Bogolioubov theorem is valid
since the corresponding Fokker-Planck equation on $\mathbb R^d\times \mathcal P_2(\mathbb R^d)$ is linear,
which allows us to establish ergodicity for MVSDEs.
\item Liu and Ma \cite{LM} proved ergodicity under integrable Lyapunov conditions, which is independent of distribution variable. Wang \cite{Wang_18} showed ergodicity under monotone conditions. Note that the Krylov-Bogolioubov theorem is invalid in their situation since the corresponding Fokker-Planck equation is nonlinear.
\end{enumerate}
\end{rem}

\subsection{Exponential ergodicity under Wasserstein quasi-distance}
Unlike the above results, we get exponential ergodicity for the following coupled SDEs under the Lyapunov condition
in which the Lyapunov function depends only on space variable
\begin{align}
\left\{
\begin{aligned}\label{MVSDE1}
&dX_t= b(X_t,\mathcal{L}_{X_t})dt+ \sigma(X_t,\mathcal{L}_{X_t})dW_t,~\mathcal L_{X_0}=\mu,\\
&d\bar X_t= \bar b(\bar X_t,\mathcal{L}_{X_t})dt+ \bar \sigma(\bar X_t,\mathcal{L}_{X_t})dW_t,~\bar X_0=x.
\end{aligned}
\right.
\end{align}
Before giving the main results in this subsection, we define Wasserstein quasi-distance $W_V,W_{V,V}$ as follows:
\begin{align*}
W_V(\mu,\nu)&:=\inf_{\pi\in\mathcal C(\mu,\nu)}\int V(x-y)\pi(dx,dy),\\
W_{V,V}(\pi_1,\pi_2)&:=\inf_{\pi\in\mathcal C(\pi_1,\pi_2)}\int \big[V(x-y)+W_V(\mu,\nu)\big]\pi\big((dx,d\mu),(dy,d\nu)\big),
\end{align*}
where $V$ is a nonnegative function defined on $\mathbb R^d$ satisfying $V(x)=0$ if and only if $x= 0$,
$\mathcal C(\mu,\nu)$ denotes the set of all couplings of $\mu,\nu$,
and $\mathcal C(\pi_1,\pi_2)$ denotes the set of all couplings of $\pi_1,\pi_2$.
Note that $W_V$ is generally not a distance since the triangle inequality may not hold.
However, it is complete in $\mathcal P(\mathbb R^d)$,
i.e. any Cauchy sequence $\{\pi_n\}\subset\mathcal P(\mathbb R^d)$
is convergent in $W_V$.
Moreover, if $(x,y)\mapsto V(x,y)$ is a distance, $W_V$ is a distance.
Furthermore, $W_{V, V}$ has the similar property with
$\mathcal P(\mathbb R^d),W_V$ replaced by $\mathcal P(\mathbb R^d\times \mathcal P_2(\mathbb R^d)),W_{V,V}$ respectively.
\begin{thm}\label{InM2}
Assume that (H1)--(H5) hold, and that there exists a nonnegative function $V\in C^{2}(\mathbb R^d)$
satisfying $V(x)=0$ if and only if $x= 0$
such that there are constants $\gamma,\bar \gamma, \beta, \bar \beta\ge 0$ with $\gamma>\beta,\bar\gamma>\bar\beta$ satisfying
\begin{align*}
(L V)(x,y,\mu,\nu)\le{} -\gamma V(x-y)+\beta W_V(\mu,\nu),\\
(\bar L V)(x,y,\mu,\nu)\le{} -\bar\gamma V(x-y)+\bar\beta W_{V}(\mu,\nu),
\end{align*}
where $LV$ is defined by
\begin{align*}
(L V)(x,y,\mu,\nu):=&(b(x, \mu)-b(t,y, \nu))\partial V(x-y)
 + \frac{1}{2}tr(\partial^2 V(x-y) A(x, y, \mu, \nu)),
\end{align*}
with $A(t,x, y, \mu, \nu)=(\sigma(x, \mu)-\sigma(y,\nu))(\sigma(x, \mu)-\sigma(y,\nu))^\top$,
and $\bar LV$ is defined similar to $LV$ in which $b,\sigma$ are replaced by $\bar b,\bar \sigma$ respectively.
Assume further that there exists $\nu_0 \in \mathcal P _{V}:=\{\mu\in\mathcal P(\mathbb R^d):\mu(V)<\infty\}$ such that
\begin{align}\label{bounded}
\sup_{t\geq0} W_{V}\left( P_t^*\nu_0, \nu_0 \right)<\infty.
\end{align}
Then there exists a unique invariant measure $\nu_{\mathcal I}\times\delta_{\{\mu_{\mathcal I}\}}$ for $P_t^*$ satisfying
\begin{align*}
&W_{V, V}(P_t^*(\nu\times\delta_{\{\mu\}}),\nu_{\mathcal I}\times\delta_{\{\mu_{\mathcal I}\}})\\
\le{}&
W_V(\nu,\nu_{\mathcal I})e^{-\bar\gamma t}
+
W_V(\mu,\mu_{\mathcal I})\bigg(e^{-(\gamma-\beta) t}+ \frac {\bar \beta}{\bar\gamma+\beta-\gamma}(e^{(\beta-\gamma)t}-e^{\bar\gamma t})\bigg)
\end{align*}
for $\mu,\nu\in\mathcal P_V$,
where if $\bar\gamma+\beta-\gamma=0$,
$$\frac {\bar \beta(e^{(\bar\gamma+\beta-\gamma)t}-1)}{\bar\gamma+\beta-\gamma}\cdot e^{-\bar\gamma t}
:=
\bar \beta te^{-\bar\gamma t}.$$
\end{thm}
\begin{proof}
By \cite[Theorem 4.1]{LM}, there exists a unique invariant measures $\mu_{\mathcal I}$ for the equation
\begin{equation*}
dX_t= b(X_t,\mathcal{L}_{X_t})dt+ \sigma(X_t,\mathcal{L}_{X_t})dW_t
\end{equation*}
such that
\begin{equation}\label{IN4}
W_V(\mathcal{L}_{X_t},\mu_{\mathcal I})\le e^{-(\gamma-\beta)t}W_V(\mathcal{L}_{X_0},\mu_{\mathcal I}).
\end{equation}
Similar to the proof of \cite[Theorem 4.1]{LM}, there is a unique invariant measure $\nu_{\mathcal I}$ for the SDE
\begin{equation}\label{Eq1}
d\bar X_t= \bar b(\bar X_t,\mu_{\mathcal I})dt+ \bar \sigma(\bar X_t,\mu_{\mathcal I})dW_t.
\end{equation}
By \cite[Theorem 4.1]{Villani}, there exist two random variables $\bar X_0^{\nu_{\mathcal I}},\bar X_0^{\nu}$ such that
$\mathcal L_{\bar X_0^{\nu_{\mathcal I}}}=\nu_{\mathcal I}, \mathcal L_{\bar X_0^{\nu}}=\nu$ and
$$
EV(\bar X_0^{\nu_{\mathcal I}}-\bar X_0^{\nu})=W_V(\mathcal L_{\bar X_0^{\nu_{\mathcal I}}},\mathcal L_{\bar X_0^{\nu}}).
$$
Denote the solutions of equation \eqref{Eq1} and the following SDE
with initial value $\bar X_0^{\nu_{\mathcal I}},\bar X_0^{\nu}$ by $\bar X_t^{\nu_{\mathcal I}},\bar X_t^{\nu}$ respectively
\begin{align*}
d\bar X_t= \bar b(\bar X_t,\mathcal{L}_{X_t})dt+ \bar \sigma(\bar X_t,\mathcal{L}_{X_t})dW_t.
\end{align*}
By It\^o's formula, we have
\begin{align*}
d V(\bar X_t^{\nu}-\bar X_t^{\nu_{\mathcal I}})
=&{}
\big(\bar b(\bar X_t^{\nu},\mathcal L_{X_t})-\bar b(\bar X_t^{\nu_{\mathcal I}},\mu_{\mathcal I})\big)\cdot\partial V(\bar X_t^{\nu}-\bar X_t^{\nu_{\mathcal I}})\\
&+{}
\frac 12tr\bigg(\partial^2 V(\bar X_t^{\nu}-\bar X_t^{\nu_{\mathcal I}})\big(\bar\sigma(\bar X_t^{\nu},\mathcal L_{X_t})-\bar\sigma(\bar X_t^{\nu_{\mathcal I}},\mu_{\mathcal I})\big)\big(\bar\sigma(\bar X_t^{\nu},\mathcal L_{X_t})-\bar\sigma(\bar X_t^{\nu_{\mathcal I}},\mu_{\mathcal I})\big)^{\top}\bigg)\\
&+{}
\big\langle\partial V(\bar X_t^{\nu}-\bar X_t^{\nu_{\mathcal I}}),\big(\bar\sigma(\bar X_t^{\nu},\mathcal L_{X_t})-\bar\sigma(\bar X_t^{\nu_{\mathcal I}},\mu_{\mathcal I})\big)dW_t\big\rangle\\
=&{}
(\bar LV)(\bar X_t^{\nu},\bar X_t^{\nu_{\mathcal I}},\mathcal L_{X_t},\nu_{\mathcal I})dt
+\big\langle\partial V(\bar X_t^{\nu}-\bar X_t^{\nu_{\mathcal I}}),\big(\bar\sigma(\bar X_t^{\nu},\mathcal L_{X_t})-\bar\sigma(\bar X_t^{\nu_{\mathcal I}},\mu_{\mathcal I})\big)dW_t\big\rangle.
\end{align*}
Taking expectation on both sides, we get
\begin{align*}
EV(\bar X_t^{\nu}-\bar X_t^{\nu_{\mathcal I}})
&={}
EV(\bar X_0^{\nu}-\bar X_0^{\nu_{\mathcal I}})+ E\int_0^t(\bar LV)(\bar X_s^{\nu},\bar X_s^{\nu_{\mathcal I}},\mathcal L_{X_s},\nu_{\mathcal I})ds\\
&\le{}
EV(\bar X_0^{\nu}-\bar X_0^{\nu_{\mathcal I}})+ E\int_0^t-\bar\gamma V(\bar X_s^{\nu}-\bar X_s^{\nu_{\mathcal I}})+\bar\beta W_V(\mathcal L_{X_s},\mu_{\mathcal I})ds\\
&\le{}
W_V(\nu,\nu_{\mathcal I})+E\int_0^t-\bar\gamma V(\bar X_s^{\nu}-\bar X_s^{\nu_{\mathcal I}})ds+\bar\beta\int_0^te^{-(\gamma-\beta)s} W_V(\mu,\mu_{\mathcal I})ds
\end{align*}
where the second inequality holds by estimate \eqref{IN4}.
Applying Gronwall's inequality, we obtain
\begin{align*}
EV(\bar X_t^{\nu}-\bar X_t^{\nu_{\mathcal I}})
&\le{}
\bigg(W_V(\nu,\nu_{\mathcal I})+\bar\beta\int_0^te^{-(\gamma-\beta)s}e^{\bar \gamma s} W_V(\mu,\mu_{\mathcal I})ds\bigg)\cdot e^{-\bar\gamma t}\\
&\le{}
\bigg(W_V(\nu,\nu_{\mathcal I})+ \frac {\bar \beta}{\bar\gamma+\beta-\gamma}(e^{(\bar\gamma+\beta-\gamma)t}-1)W_V(\mu,\mu_{\mathcal I})\bigg)\cdot e^{-\bar\gamma t},
\end{align*}
where if $\bar\gamma+\beta-\gamma=0$,
$$\frac {\bar \beta(e^{(\bar\gamma+\beta-\gamma)t}-1)}{\bar\gamma+\beta-\gamma}\cdot e^{-\bar\gamma t}
:=
\lim_{a\to \gamma}\frac {\bar \beta(e^{(a-\gamma)t}-1)}{a-\gamma}\cdot e^{-\bar\gamma t}
=\bar \beta te^{-\bar\gamma t}.$$
Therefore, we have
\begin{align*}
&W_{V,V}(P_t^*(\nu\times\delta_{\{\mu\}}),\nu_{\mathcal I}\times\delta_{\{\mu_{\mathcal I}\}})
={}
W_{V,V}(\mathcal L_{\bar X_t^{\nu}}\times\delta_{\{\mathcal L_{X_t^{\mu}}\}},\nu_{\mathcal I}\times\delta_{\{\mu_{\mathcal I}\}})\\
\le{}&
W_{V}(\mathcal L_{\bar X_t^{\nu}},\nu_{\mathcal I})+W_V(\mathcal L_{X_t^{\mu}},\mu_{\mathcal I})\\
\le{}&
W_V(\mu_{\mathcal I},\mu)e^{-(\gamma-\beta) t}+
\bigg(W_V(\nu,\nu_{\mathcal I})+\frac {\bar \beta}{\bar\gamma+\beta-\gamma}(e^{(\bar\gamma+\beta-\gamma)t}-1)W_V(\mu,\mu_{\mathcal I})\bigg) e^{-\bar\gamma t}\\
={}&
W_V(\nu,\nu_{\mathcal I})e^{-\bar\gamma t}
+
W_V(\mu,\mu_{\mathcal I})\bigg(e^{-(\gamma-\beta) t}+ \frac {\bar \beta}{\bar\gamma+\beta-\gamma}(e^{(\beta-\gamma)t}-e^{-\bar\gamma t})\bigg),
\end{align*}
where the first equality holds by the definition of $P_t^*$.
\end{proof}
\begin{rem}
\begin{enumerate}
\item The condition \eqref{bounded} means that there is a ``bounded orbit" in $\mathcal P_{V}$, which is necessary and natural because the system cannot have an invariant measure if any orbit is unbounded.
\item By taking $V(\cdot)= |\cdot|^2$, our result Theorem \ref{InM2} reduces to Theorem 6.1 in Ren at al. \cite{RRW}.
\item If $b=\bar b$ and $\sigma=\bar\sigma$, we have $\mu_{\mathcal I}=\nu_{\mathcal I}$.
Moreover, if $\mu=\nu$, our result Theorem \ref{InM2} reduces to Theorem 4.1 in Liu and Ma \cite{LM}.
\end{enumerate}
\end{rem}

\section{Applications}
In this section, we give some examples to illustrate our theoretical results.
\begin{exam}
For any $x\in\mathbb R,\mu\in \mathcal P(\mathbb R)$, let
\begin{equation*}
b(x,\mu)=-x\int_{\mathbb R}y^2\mu(dy), ~\sigma(x,\mu)= \sqrt 2 x.
\end{equation*}
Then the following SDE
\begin{equation*}
dX_t= b(X_t,\mathcal{L}_{X_t})dt+ \sigma(X_t,\mathcal{L}_{X_t})dW_t
\end{equation*}
has a unique solution for any given initial value $X_0\in L^2(\Omega,\mathcal F,P)$.
\end{exam}
\begin{proof}
(H1), (H3) and (H5) obviously hold. Now we prove that (H2) and (H4) hold.
For any $x\in\mathbb R, \mu\in \mathcal P(\mathbb R)$, let $V(x,\mu)=x^4+\int_{\mathbb R}y^6\mu(dy)$. Then (H4) holds and
\begin{align*}
\partial _xV(x,\mu)&={}4x^3, ~ \partial _x^2V(x,\mu)={}12x^2,\\
\partial _{\mu}V(x,\mu)(z)&={}6z^5, ~ \partial _z\partial _{\mu}V(x,\mu)(z)=30z^4.
\end{align*}
Thus, we have
\begin{align*}
LV(x,\mu)
={} &
-x\int y^2\mu(dy)\cdot 4x^3+ \frac 12 \cdot 2x^2\cdot 12x^2\\
& +
\int\bigg((-z\int y^2\mu(dy)) \cdot 6z^5+\frac 12 \cdot 2z^2\cdot 30z^4 \bigg)\mu(dz)\\
\le{} &
12x^4+ \int 30z^6 \mu(dz)
\le 30 V(x,\mu),
\end{align*}
i.e. (H2) holds.
Therefore, by Theorem \ref{Thm1}, there exists a unique solution.
\end{proof}

\begin{exam}\label{Ex5.2}
For any $x\in\mathbb R, \mu\in \mathcal P(\mathbb R)$, let
\begin{equation*}
b(x,\mu)=-3x+3\int_{\mathbb R} y\mu(dy), ~\sigma(x,\mu)= x-\int_{\mathbb R} y\mu(dy).
\end{equation*}
Then the following SDE
\begin{equation*}
dX_t= b(X_t,\mathcal{L}_{X_t})dt+ \sigma(X_t,\mathcal{L}_{X_t})dW_t
\end{equation*}
has a unique solution for any given initial value $X_0\in L^2(\Omega,\mathcal F,P)$.
Furthermore, there exists a unique invariant measure.
\end{exam}
\begin{proof}
The conditions (H1), (H3) and (H5) obviously hold.
Let $V(x,\mu)=\frac 14\big(x-\int_{\mathbb R} y\mu(dy)\big)^4$. Then it is immediately to see that (H4) holds. Note that we have
\begin{align*}
\partial _xV(x,\mu)&={}\big(x-\int y\mu(dy)\big)^3, ~ \partial _x^2V(x,\mu)={}3\big(x-\int y\mu(dy)\big)^2,\\
\partial _{\mu}V(x,\mu)(z)&={}-\big(x-\int y\mu(dy)\big)^3, ~ \partial _z\partial _{\mu}V(x,\mu)(z)=0,
\end{align*}
and
\begin{align*}
\mathcal LV(x,\mu)
=&{}
-3\big(x-\int y\mu(dy)\big)^4+\frac 32\big(x-\int y\mu(dy)\big)^4\\
&+{}
\int 3\big(z-\int y\mu(dy)\big)\big(x-\int y\mu(dy)\big)^3\mu(dz)\\
=&{}
-\frac 32\big(x-\int y\mu(dy)\big)^4
=
-\frac 32V(x,\mu).
\end{align*}
That is to say that (H2) holds.
Therefore, by Theorem \ref{Thm1} and Theorem \ref{Thm2}, we obtain the desired results.
\end{proof}
\begin{exam}
For any $x\in\mathbb R, \mu\in \mathcal P(\mathbb R)$, let
\begin{align*}
b(x,\mu)&={}-3x+3\int_{\mathbb R} y\mu(dy), ~\sigma(x,\mu)={} x-\int_{\mathbb R} y\mu(dy),\\
\bar b(x,\mu)&={}-x^3 -2\int_{\mathbb R}(x+ \frac 12 y)\mu(dy), ~\bar\sigma(x,\mu)={}\int_{\mathbb R}(x+ \frac 12 y)\mu(dy).
\end{align*}
Then the semigroup $P_t^*$ is exponentially ergodic for the SDEs:
\begin{align}
\left\{
\begin{aligned}\label{MVSDE2}
&dX_t={} b(X_t,\mathcal{L}_{X_t})dt+ \sigma(X_t,\mathcal{L}_{X_t})dW_t,~\mathcal L_{X_0}=\mu,\\
&d\bar X_t={} \bar b(\bar X_t,\mathcal{L}_{X_t})dt+ \bar \sigma(\bar X_t,\mathcal{L}_{X_t})dW_t,~\bar X_0=x.
\end{aligned}
\right.
\end{align}
\end{exam}
\begin{proof}
By Example \ref{Ex5.2} and \cite[Example 5.1]{LM}, the eqations \eqref{MVSDE2} has a unique solution, and satisfies $$(\bar LV)(x,y,\mu,\nu)\le -2V(x-y)+\frac12 W_V(\mu,\nu),$$ with $V(\cdot)=|\cdot|^2$.
We next show that $LV(x-y)\le -3V(x-y)+3 W_V(\mu,\nu)$.
By the Cauchy-Schwarz inequality and Kantorovich duality, see for instance \cite[Theorem 5.10]{Villani}, we have
\begin{align*}
(LV)(x,y,\mu,\nu)
=&{}
-6(x-y)\big(x-y-\int z\mu(dz)+\int z\nu(dz)\big)+\big(x-y-\int z\mu(dz)+\int z\nu(dz)\big)^2\\
=&{}
-5(x-y)^2-4(x-y)\big(-\int z\mu(dz)+\int z\nu(dz)\big)+ \big(-\int z\mu(dz)+\int z\nu(dz)\big)^2\\
\le&{}
-3(x-y)^2+3\big(-\int z\mu(dz)+\int z\nu(dz)\big)^2\\
\le&{}
-3V(x-y)+3W_V(\mu,\nu).
\end{align*}
The result now follows from Theorem \ref{InM2}.
\end{proof}

\section*{Acknowledgement}

This work is partially supported by NSFC Grants 11871132, 11925102, Dalian High-level Talent Innovation Project (Grant 2020RD09).

\end{document}